\documentclass[10pt,a4paper]{article}

\usepackage[margin=1in]{geometry}
\usepackage{fourier}
\usepackage{tikz,tkz-berge,pgfplots}
\usepackage{algorithm,algorithmic}
\usepackage{amsmath,amsthm,amssymb}

\newtheorem{theorem}{Theorem}

\newtheorem{lemma}{Lemma}
\newtheorem{proposition}{Proposition}

\DeclareMathOperator*{\corr}{Corr}

\date{}

\begin{document}

\title{Adaptively Blocked Particle Filtering with Spatial Smoothing in Large-Scale Dynamic Random Fields\thanks{This work was supported by AFOSR/AOARD via AOARD-144042.}}

\author{Francesco Bertoli \thanks{F. Bertoli is with the Australian National University and NICTA. He is supported by NICTA.} \and Adrian N. Bishop \thanks{A.N. Bishop is with the Australian National University and NICTA. He is supported by NICTA and the Australian Research Council (ARC) via a Discovery Early Career Researcher Award (DE-120102873).}}

\maketitle

\begin{abstract} 
The typical particle filtering approximation error is exponentially dependent on the dimension of the model. Therefore, to control this error, an enormous number of particles are required, which means a heavy computational burden that is often so great it is simply prohibitive. Rebeschini and van Handel (2013) consider particle filtering in a large-scale dynamic random field. Through a suitable localisation operation, they prove that a modified particle filtering algorithm can achieve an approximation error that is mostly independent of the problem dimension. To achieve this feat, they inadvertently introduce a systematic bias that is spatially dependent (in that the bias at one site is dependent on the location of that site). This bias consequently varies throughout field. In this work, a simple extension to the algorithm of Rebeschini and van Handel is introduced which acts to average this bias term over each site in the field through a kind of spatial smoothing. It is shown that for a certain class of random field it is possible to achieve a completely spatially uniform bound on the bias and that in any general random field the spatial inhomogeneity is significantly reduced when compared to the case in which spatial smoothing is not considered. While the focus is on spatial averaging in this work, the proposed algorithm seemingly exhibits other advantageous properties such as improved robustness and accuracy in those cases in which the underlying dynamic field is time varying.
\end{abstract}

\section{Introduction}

Particle filtering, as it applies here, is a powerful technique for (recursive) estimation and inference in nonlinear dynamical state-space models subject to stochastic influence. In theory, the state of an underlying stochastic dynamical system can be recursively estimated by composing the posterior probability of the state conditioned on the random observations as they become available and using a given initial prior and the stochastic dynamical system model. This process is known as recursive Bayesian filtering \cite{anderson:79} and it is generally intractable in practice \cite{anderson:79,ad-ndf-ng:01a}. The particle filter is an approximation of the Bayesian filter that employs random sampling to represent the posterior where such samples (or particles) are propagated in practice through (sequential) importance sampling \cite{ad-ndf-ng:01a} that attempts to capture the dynamics of the underlying system as well as the likelihood model on the observations. A resampling step inserted into the recursion is also crucial to avoid sample degeneracy and control the variance over time. To date, the particle filter has been well studied and we point to the still relevant early work \cite{gordon1993novel,kitagawa1996monte} as well as the comprehensive coverage in \cite{ad-ndf-ng:01a,cappe2005inference} for further background.

The convergence of the particle filter in time has been well considered \cite{del2001stability,crisan2002survey}. By convergence in time, we mean that one can show that the approximation error, with respect to an idealised Bayesian filter, and due mostly to the random sampling, can be controlled through time \cite{crisan2002survey}. One can even show the filter approximation error due to the sampling approximation remains bounded uniformly in time \cite{del2001stability,van2009uniform,del2013mean}. Such results provide significant grounding for the application of particle filtering in numerous application domains \cite{ad-ndf-ng:01a}.

Despite substantial analysis justifying use of the particle filter in numerous applications of interest, a limitation to date surrounds application of the particle filter in high-dimensional estimation and inference problems \cite{snyder2008obstacles,van2009particle,van2010nonlinear}. The limiting factor is computational complexity. Analysis in \cite{bickel2008sharp,quang2010insight} suggests that the particle filter approximation error is exponential in the dimension of the underlying (measurement) model while the same error is controlled by the number of samples according to something like the inverse square root. This relationship is clearly exposited in \cite{rebeschini2013comparison}. The conclusion is simply that, an enormous number of particles (exponential in the dimension) must be maintained if one is to control the estimation error at a reasonable level when applying standard particle filter implementations in a high-dimensional dynamical estimation problem. An enormous number of particles means a heavy computational burden which is often so large it is prohibitive. 

The good news is that recent studies \cite{beskos2011stability,beskos2011error,rebeschini2013can,rebeschini2013comparison} imply that high-dimensional particle filtering may be feasible in particular applications and/or if one is willing to accept a degree of systematic bias. In \cite{beskos2011stability}, the particle filter is applied in a static setting where the objective is to sample from some high-dimensional target distribution. In this case, through a sequence of intermediate and simpler distributions, it is shown that the particle filter will converge to a sampled representation of the target distribution with a typical Monte Carlo error (inverse in the number of particles) given a complexity on the order of the dimension squared. Although \cite{beskos2011stability} deals only, in essence, with a static problem of sampling from a fixed target distribution, the analysis introduces a novel way of thinking about high-dimensional particle filtering which may carry over to dynamic filtering problems. Related work appears in \cite{beskos2011error}. 

\subsection{Background: The Motivating Paper}

Rebeschini and van Handel in \cite{rebeschini2013can} consider particle filtering in large-scale dynamic random fields. They introduce a simple blocked particle filter which localises the filter to the blocks in a partition of the random field. Here, localisation means that the particle filtering prior distribution at each time is independently updated/corrected within each block through application of the observations locally conditioned on that block. The posterior over each block is independent across blocks and the posterior over the entire field is just the product of each blocked posterior. The real contribution of \cite{rebeschini2013can} is a descriptive and technical analysis that shows the error introduced due to the localisation procedure can be readily controlled if the dynamics of the random field at each site are only locally dependent on those sites within close proximity. The standard sampling approximation error is shown to be exponential in only the size of the individual blocks. The number of samples/particles controls the sampling approximation error at the typical rate while the error due to the localisation process is a systematic bias that can only be controlled through an increase in the block size. Since each block is updated independently, parallel implementation is readily applicable and the computational burden may be alleviated, albeit this remains to be seen in practice. While the results of \cite{rebeschini2013can} are at the proof-of-concept stage, the idea is incredibly powerful and it provides the entire motivation for the study herein. 

It was just noted that the error due to the localisation, or blocking, procedure discussed in \cite{rebeschini2013can} is systematic and controlled only by the block size. Actually, this is an exaggeration and the stated result is significantly more promising. To analyse the effect of the blocking procedure, the authors introduce a tool termed the decay of correlations property which captures a spatial notion of stability, i.e. which measures how quickly dependence between sites in the field decays as a function of the distance between those sites. They show that if a suitable decay of correlations exists, then the error at any site in the field introduced due to the blocking procedure alone is dependent primarily on that site's distance to the border of the block. In other words, the systematic bias introduced due to the blocking operator is small at those sites in the field far removed from the block borders, large on the borders and varies in between. The blocking bias is not spatially uniform. The error due to the sampling procedure inherent to all particle filters is still exponential in the block size and controlled by the number of particles. The sampling approximation error implies one should not seek excessively large block sizes. On the other hand, small blocks imply the spatial inhomogeneity of the total error caused by the blocking bias is exasperated. Ideally, one would like an algorithm that maintains a spatially homogeneous error within each block. One can then focus on considering, in detail, the block-size vs. error tradeoff. The authors in \cite{rebeschini2013can} devote much discussion to the spatial inhomogeneity of the blocking bias and its significance.

\subsection{Contribution}

This work details a simple, yet relevant, algorithmic adjustment to the blocked particle filter introduced in \cite{rebeschini2013can}. The idea is to spatially average the bias caused by the blocking procedure by considering an adaptive sequence of partitions over the random field instead of a single partition. This adaptive blocking procedure has the effect of spatially smoothing the error caused by any single application of the blocking/localisation operation. In a certain class of random fields, this smoothing effect leads to a completely spatially homogeneous error bound. That is, the bound on the total filtering error at any site in the field is completely independent of the location of that site. In more general random fields, this spatial averaging effect can lead to a significant reduction in the spatial inhomogeneity of the particle filter error bounds when compared with \cite{rebeschini2013can}. This smoothing effect is of practical relevance in those cases in which the blocks should be kept relatively small for computational reasons or in which the dynamics of the random field are less than trivially localised. 

The results presented here are largely at the same proof-of-concept level as those presented in the motivating paper \cite{rebeschini2013can}. The stated results in both studies are of a quantitative nature that is far from optimal. Nevertheless, the applicability of the particular algorithms is likely far less restricted than a strict reading of the results would suggest; indeed this is seemingly also true for the celebrated time-uniform particle filter convergence results \cite{del2001stability}. It is with this applicability in mind that the algorithmic extensions considered herein are proposed. The extensions considered are conceptually simple, easy to implement, and do not generally add to the computational burden of the algorithm. Beyond the important spatial smoothing effect of the proposed filter, allowing for (adapting) multiple partitions of the field may also provide algorithmic robustness in those cases in which the random field is time-varying etc as the partitions can be adapted online, or it may allow one to adaptively focus computation on certain locations of interest for periods of time etc. Other advantages of partition adaptation are envisioned.

\subsection{Paper Organization}

The paper is organised as follows. In Section 2 we introduce the model and problem setup. In Section 3 we introduce the Bayesian filtering framework, the (standard) particle filtering algorithm and we introduce the adaptively blocked particle filter for high-dimensional estimation problems. In Section 4 we state the main result, which consists of a time-average and spatially smoothed total error bound on the adaptively blocked particle filter approximation to the true Bayesian nonlinear filter. In Section 4 we also note that the bound may well be completely spatially-uniform and we outline the strategy for proving this result. In Section 5 we explore in more detail the error bound and the spatial smoothing effect of adaptively sequencing through partitions in the blocked particle filter. In Section 5 we discuss, in more detail, the class of random fields and the sequence of partitions that may lead to complete spatial uniformity in the error bound.

To this point, the problem formulation, algorithm, convergence results, and the algorithmic/convergence discussions are given. A casual reader may stop at this point, and take for granted the convergence results and the (increased) spatial homogeneity of the filtering approximation error across the random field. Indeed, the conceptual simplicity of the algorithm, and its spatial averaging property, may be sufficient to convince one of the general existence of such a convergence result (given also the results and analysis in \cite{rebeschini2013can}). Subsequently, the details of this result, as they are far from optimal, may be of lesser significance.

Going forward, in Section 6 we provide the technical analysis leading to the main result and following the proof strategy introduced earlier. The proofs required in this work largely overlap with those in \cite{rebeschini2013can} and only the required changes are derived here, with reference made to the motivating paper as often as possible. Indeed, we encourage all readers interested in high-dimensional particle filtering to study \cite{rebeschini2013can} since a very descriptive and accessible coverage of this topic and the blocked particle filter is provided therein (prior to the detailed technical analysis set out in \cite{rebeschini2013can} and to which we point as often as possible to prove our case).

\section{Problem Setup}

For simplicity and to ease comparison we borrow the problem scenario directly from \cite{rebeschini2013can}. 

Consider a Polish state space $\mathbb{X}$ with $\sigma$-algebra $\mathcal{X}$ and reference measure $\psi$, and a Polish state space $\mathbb{Y}$ with $\sigma$-algebra $\mathcal{Y}$ and a reference measure $\varphi$. Introduce on $\mathbb{X}$, a Markov chain $(X_n)_{n\geq 0}$ with a transition density $p:\mathbb{X}\times\mathbb{X}\to\mathbb{R}_+$ with respect to $\psi$. Introduce on $\mathbb{Y}$, a sequence $(Y_n)_{n\geq 0}$ that is conditionally independent given $(X_n)_{n\geq 0}$ and has a transition density $g:\mathbb{X}\times\mathbb{Y}\to\mathbb{R}_+$ with respect to $\varphi$. We interpret $(X_n)_{n\geq 0}$ as an underlying dynamical process that is observed through $(Y_n)_{n\geq 0}$. The pair $(X_n,Y_n)_{n\geq 0}$ is also a Markov chain. 

Now suppose the state $(X_n,Y_n)$ at each time $n$ is a random field $(X_n^v,Y_n^v)_{v\in V}$ indexed by a (finite) undirected graph $G=(V,E)$ where $V$ corresponds to the set of sites and $E$ corresponds to the set of edges that define the structure of the field. The dimension of the model is then at least as big as the cardinality of the vertex set $V$ and we assume this to be large. 

To be more precise, the space $\mathbb{X}$ and $\mathbb{Y}$ are of product form $\mathbb{X} = \prod_{v\in V}\mathbb{X}^v$ and $\mathbb{Y} = \prod_{v\in V}\mathbb{Y}^v$ respectively. The associated reference measures then given by $\psi = \bigotimes_{v\in V}\psi^v$ and $\varphi = \bigotimes_{v\in V}\varphi^v$ where $\psi^v$ and $\varphi^v$ are reference measures on $\mathbb{X}^v$ and $\mathbb{Y}^v$ respectively. The transition densities $p$ and $g$ are given by
$$
	p(x,z) = \prod_{v\in V}p^v(x,z^v), \quad g(x,y) = \prod_{v\in V}g^v(x^v,y^v)
$$
where $p^v:\mathbb{X}\times\mathbb{X}^v\to\mathbb{R}_+$ and $g^v:\mathbb{X}^v\times\mathbb{Y}^v\to\mathbb{R}_+$ are defined with respect to $\psi^v$ and $\varphi^v$ respectively. 

The model assumes the observations $(Y_n)_{n\geq 0}$ are completely local in the sense that $g^v(x^v,y^v)$ depends only on $x^v$ or, in other words, the conditional distribution of $Y_n^v$ given $X_n$ depends only on $X_n^v$. 

A distance $d(v,v')$, that counts hops along the shortest path between $v,v'\in V$, is associated with $G=(V,E)$. Now for a fixed $r\in\mathbb{N}$ and for each vertex $v\in V$ we define $N(v)=\{v'\in V:d(v,v')\leq r\}$ which specifies a neighbourhood of $v$. We then assume the dynamics of $(X_n)_{n\geq 0}$ are local in the sense that $p^v(x,z^v)$ depends only on $x^{N(v)}$ or in other words the conditional distribution of $X_n^v$ given $X_0,\ldots,X_{n-1}$ depends only on $X_{n-1}^{N(v)}$. More precisely, the dynamics obey $p^v(x,z^v)=p^v(\tilde x,z^v)$ whenever $x^{N(v)}=\tilde x^{N(v)}$ where $x^J=(x^j)_{j\in J}$ for $J\subseteq V$. 

We refer to the motivating paper \cite{rebeschini2013can} for further discussion on such models and the references therein for background of where such models appear in the literature. Also, see \cite{vanmarcke2010random} for such modelling motivation.

Since the process $(X_n)_{n\geq 0}$ is not directly observable, the filtering problem of interest is one of recursively estimating the unobserved state $X_n$ given the observation history $Y_1,\ldots,Y_n$. That is, the filtering problem is one of computing
$$
	\pi_n^\mu \triangleq \mathbf{P}^\mu[X_n\in\,\cdot\,|Y_1,\ldots,Y_n]
$$
where $\mathbf{P}^\mu$ is the probability measure under which $(X_n,Y_n)_{n\geq 0}$ is a Markov chain with a transition probability $P$ that can be factored as $P((x,y),A) = \int \mathbf{1}_A(x',y') p(x,x') g(x',y') \psi(dx') \varphi(dy')$ for any $A\in\mathcal{X}\times\mathcal{Y}$ and where the initial condition $X_0\sim\mu$ is an arbitrary probability measure $\mu$ on $\mathbb{X}$. Notationally, $\pi_n^x \triangleq\pi_n^{\delta(x)}$.

\section{Adaptively Blocked Particle Filtering}

Firstly, we outline the ideal nonlinear recursive Bayes filter, followed by the standard bootstrap particle filter. We note briefly the computational problem involved in applying the bootstrap filter to high-dimensional estimation problems. We then outline our adaptively blocked particle filter and note its straightforward relationship to the algorithm of \cite{rebeschini2013can} and discuss generally the motivation for this algorithm as it applies to filtering of high-dimensional systems.

It is well known that through an application of Bayes rule, the filter $\pi_n^\mu$ can be computed recursively via
$$
	\pi_0^\mu=\mu,\quad \pi_n^\mu = \mathsf{F}_n\pi_{n-1}^\mu\quad(n\geq 1)
$$
where, it is common for practical, as well as conceptual, reasons to define $\mathsf{F}_n\triangleq\mathsf{C}_n\mathsf{P}$ where
$$
	(\mathsf{P}\rho)(f) \triangleq \int f(x')\,p(x,x')\,\psi(dx')\,\rho(dx)
$$
is a prediction in which the filter estimate $\pi_{n-1}^\mu$ is propagated forward using the dynamics of the underlying process $(X_n)_{n\geq 0}$ and
$$	
	(\mathsf{C}_n\rho)(f) \triangleq \frac{\int f(x)\,g(x,Y_n)\,\rho(dx)}{\int g(x,Y_n)\,\rho(dx)}
$$
is a so-called correction (or update) in which the predicted distribution is updated by conditioning it on the observation $Y_n$ to obtain $\pi_n^\mu$. Graphically, 
$$
	\pi_{n-1}^\mu ~ \xrightarrow{\text{prediction}} ~ \pi_{n-}^\mu = \mathsf{P}\pi_{n-1}^\mu ~ \xrightarrow{\text{correction}} ~ \pi_n^\mu = \mathsf{C}_n\pi_{n-}^\mu
$$
The recursive structure of the filter allows estimation of the underlying dynamic process to be carried out `on-line' over a long time horizon and incorporating measurements as they become available. However, it is well known that to this point such a filter is impractical since at the level of arbitrary probability measures it must be considered of infinite dimension. In general, no exact finite dimensional nonlinear filter can be computed. 

One approximation to the nonlinear filter employs sampling and Monte Carlo approximation. Let $N\geq 1$ denote the number of samples (or particles) used in the approximation and define $\mathsf{S}^N$ to be the sampling operator which computes a random measure
$$
	\mathsf{S}^N\rho \triangleq \frac{1}{N}\sum_{i=1}^N\delta_{x(i)}, \quad x(i)\mathrm{~is~i.i.d.}\sim\rho
$$
with respect to some probability measure $\rho$. This random measure is a discrete approximation of $\rho$ and converges to $\rho$ with $N$ at a typical rate of $1/\sqrt{N}$.

The most common and arguably the simplest Monte Carlo approximation of nonlinear filtering is given by
$$
	\hat\pi_0^\mu=\mu, \quad \hat\pi_n^\mu = \mathsf{\hat F}_n\hat\pi_{n-1}^\mu\quad(n\geq 1)
$$
where $\mathsf{\hat F}_n\triangleq\mathsf{C}_n\mathsf{S}^N\mathsf{P}$ now consists of three operations
$$
	\hat\pi_{n-1}^\mu ~ \xrightarrow{\text{prediction}} ~ \mathsf{P}\hat\pi_{n-1}^\mu~ \xrightarrow{\text{sampling}} ~ \hat\pi_{n-}^\mu = \mathsf{S}^N\mathsf{P}\hat\pi_{n-1}^\mu ~ \xrightarrow{\text{correction}} ~ \hat\pi_n^\mu = \mathsf{C}_n\hat\pi_{n-}^\mu
$$
 This recursion yields the bootstrap particle filtering algorithm \cite{gordon1993novel}. This algorithm is simple to implement and $\hat\pi_n^\mu$ converges to the exact filter $\pi_n^\mu$ as $N\to\infty$. Resampling and other operations are typically incorporated into the bootstrap particle filter to improve performance \cite{ad-ndf-ng:01a}.

If the (standard) bootstrap particle filter is applied to a system of dimension $|V|$, then, typically \cite{snyder2008obstacles,bickel2008sharp,quang2010insight}, the approximation error is exponential in $|V|$ and inversely proportional to something like $\sqrt{N}$. If $|V|$ is large, then one needs a huge number of particles $N$ to achieve a desired error rate and this requires a heavy computational burden. In many applications, like target tracking \cite{ad-ndf-ng:01a}, there are typically no computational barriers to achieving an acceptable error. In large-scale, high-dimensional, estimation problems the particle filter is often computationally infeasible which motivates the study in \cite{rebeschini2013can} and obviously in this work. 

To this end, we introduce a partition $\mathcal{K}$ of the vertex set $V$ into non-overlapping blocks
$$
	K\cap K'=\emptyset, ~K\ne K', ~K,K'\in\mathcal{K} ~\mathrm{and} ~V = \bigcup_{K\in\mathcal{K}}K
$$
Now suppose there exists a finite number $m\in\mathbb{N}$ of partitions $\mathcal{K}_0, \dots, \mathcal{K}_{m-1}$ of this type. There exists a non-negative constant $\theta\in \mathbb{R}_+$ and a positive $\vartheta\in \mathbb{R}_+$ such that given a positive $\beta\in\mathbb{R}_+$ then for every node $v\in V$ we have
$$
\theta \leq \theta_m(v) = \frac{1}{m}\sum_{j=0}^{m-1} d(v,\partial K_j(v)) \quad \mathrm{and} \quad 0 < \vartheta_m(v) = \frac{1}{m}\sum_{j=0}^{m-1} e^{-\beta d(v,\partial K_j(v))} \leq \vartheta
$$
\noindent where $K_j(v)\in\mathcal{K}_j$ and we write $K(v)$ when $v\in K$ for some $K$ in some $\mathcal{K}$. Here, $\theta$ is the smallest average distance between any site and the borders $\partial K_j(v) = \{v'\in K_j(v):N(v')\not\subseteq K_j(v)\}$ of those blocks containing it, while $\vartheta$ captures a similar property in a more round about manner. We now define \cite{rebeschini2013can} the blocking operator for some partition $\mathcal{K}$
$$
	\mathsf{B}(\mathcal{K})\rho \triangleq \bigotimes_{K\in\mathcal{K}}\mathsf{B}^K\rho
$$
where for any measure $\rho$ on $\mathbb{X}=\bigotimes_{v\in V}\mathbb{X}^v$ and $J\subseteq V$ we denote by $\mathsf{B}^J\rho$ the marginal of $\rho$ on $\bigotimes_{v\in J}\mathbb{X}^v$. The random field described by the measure $\mathsf{B}(\mathcal{K})\rho$ on $\mathbb{X}$ is independent across blocks defined by the partition $\mathcal{K}$.

The adaptively blocked particle filter adds a blocking operation into the bootstrap particle filter recursion
$$
	\hat\pi_0^\mu=\mu, \quad \hat\pi_n^\mu = \mathsf{\hat F}_n\hat\pi_{n-1}^\mu\quad(n\geq 1)
$$
where $\mathsf{\hat F}_n\triangleq\mathsf{C}_n\mathsf{B}(\mathcal{K}_{\sigma(n)})\mathsf{S}^N\mathsf{P}$ consists of four operations
$$
	\hat\pi_{n-1}^\mu ~ \xrightarrow{\text{prediction/sampling}} ~ \hat\pi_{n-}^\mu = \mathsf{S}^N\mathsf{P}\hat\pi_{n-1}^\mu ~ \xrightarrow{\text{blocking/correction}} ~ \hat\pi_n^\mu = \mathsf{C}_n\mathsf{B}(\mathcal{K}_{\sigma(n)})\hat\pi_{n-}^\mu
$$
\noindent where $\sigma:\mathbb{N}\rightarrow\{0,\ldots,m-1\}$ is a partition switching signal. If $m=1$ then the adaptively blocked particle filter reduces to the blocked particle filter considered in \cite{rebeschini2013can}. If $\mathcal{K}_0=\ldots=\mathcal{K}_{m-1}=\{V\}$ then the adaptively blocked particle filter reduces to the bootstrap particle filter. The resulting algorithm is given in Algorithm \ref{alg:filter}.

\begin{algorithm} 
\caption{Adaptively Blocked Particle Filter} \label{alg:filter}
\begin{algorithmic}
\vspace{0.1cm}
\STATE consider the partitions $\mathcal{K}_0, \dots, \mathcal{K}_{m-1}$ 
\vspace{0.2cm}
\STATE let $\hat\pi_0^\mu = \mu$ 
\vspace{0.2cm}
\FOR{$k=1,\ldots,n$} 
	\vspace{0.2cm}
	\STATE resample i.i.d. $\hat x_{k-1}(i) \sim \hat\pi_{k-1}^\mu$, $i=1,\ldots,N$ 
	\vspace{0.2cm}
	\STATE sample $x_k^v(i) \sim p^v(\hat x_{k-1}(i), x^v)\psi(dx^v)$, $i=1,\ldots,N$, $\forall v\in V$ 
	\vspace{0.2cm}
	\STATE compute $w_k^K(i)={\prod_{v\in K}g^v(x_k^v(i),Y_k^v)}\bigg/{\sum_{j=1}^N \prod_{v\in K}g^v(x_k^v(j),Y_k^v)}$, $i=1,\ldots,N$, $\forall K\in\mathcal{K}_{\sigma(k)}$ 
	\vspace{0.2cm}
	\STATE let $\hat \pi_k^\mu = \bigotimes_{K\in\mathcal{K}_{\sigma(k)}} \sum_{i=1}^Nw_k^K(i)\,\delta_{x_k^K(i)}$
	\vspace{0.2cm}
\ENDFOR
\vspace{0.1cm}
\end{algorithmic}
\end{algorithm}

The only difference between the adaptively blocked particle filter and the block particle filter of \cite{rebeschini2013can} is the adaptive consideration of multiple field partitions during the execution of the algorithm. Going forward, the notation $\hat\pi_n^\mu$ will refer to the adaptively blocked particle filter of Algorithm \ref{alg:filter}.

At any time $n$, only measurements in block $K\in\mathcal{K}_{\sigma(n)}$ are used to update the filter in block $K$. Each block $K\in\mathcal{K}_{\sigma(n)}$ can therefore be updated in parallel. The complexity of updating each block with a given error is thus dependent only on the cardinality of that block and not on the dimension of the entire random field. If the additional error for the entire filter caused by the blocking approximation can be sufficiently controlled, it seems the curse of dimensionality as it applies to the particle filter in general can be alleviated.  

It is clear that at any given iteration the blocking operator decouples the distribution at the boundaries of the blocks. In the motivating paper \cite{rebeschini2013can}, only a single partition is considered and it follows that the filtering approximation error will be larger at those vertices close to the boundary of each block than at those sites toward the centre of each block. The result is a spatially non-homogeneous filtering error and indeed the error bound derived in the motivating paper \cite{rebeschini2013can} is spatially dependent. A comprehensive and insightful discussion on this problem is provided in \cite{rebeschini2013can}. By adaptively applying different partitions one may ensure that, on average, each site of the random field is (at least approximately) the same distance from the borders of the blocks during some cycle. One would hope then that the error is, on average, spatially homogeneous and that one can achieve an error bound that is site independent. We note with this in mind the important special case of Algorithm \ref{alg:filter} in which $\sigma(n)=n\,(\mathrm{mod}\,m)$ or where the partitions are applied in a cyclical order. 

We stress that beyond the possibility of spatial smoothing, allowing for multiple partitions of the field may also provide algorithmic robustness in those cases in which the random field is time-varying since the partitions can be adapted online, or it may allow one to adaptively focus computation on certain locations of interest for periods of time etc. The design of $\sigma$ offers seemingly much flexibility and other advantages of partition adaptation are envisioned but the discussion here will focus on spatial averaging of the blocking bias.

\section{The Main Results}

As in \cite{rebeschini2013can} we define the following norm 
$$
	\VERT\rho-\rho'\VERT_J \triangleq \sup_{f\in\mathfrak{X}^J: |f| \leq 1}\mathbf{E} \left[|\rho(f)-\rho'(f)|^2 \right]^{1/2}
$$
between two random measures $\rho$ and $\rho'$ on $\mathbb{X}$ where $\mathfrak{X}^J$ is the class of all measurable functions $f:\mathbb{X}\rightarrow\mathbb{R}$ with $f(x)=f(\tilde x)$ whenever $x^J=(x^j)_{j\in J}$ for $J\subseteq V$. For example, one then finds $\VERT\rho-\mathsf{S}^N\rho\VERT_{V} \leq 1/\sqrt{N}$ which captures the typical Monte Carlo approximation error. The goal is to bound the error between the nominal (ideal) Bayesian filter $\pi_n^\mu$ and the adaptively blocked particle filter $\hat\pi_n^\mu$. Recall that both the ideal filter and the adaptively blocked particle filter are defined recursively
$$
	\pi_n^\mu= \mathsf{F}_n\cdots\mathsf{F}_1\mu, \quad \hat\pi_n^\mu=\mathsf{\hat F_n}\cdots\mathsf{\hat F}_1\mu
$$
where $\mathsf{F}_n\triangleq\mathsf{C}_n\mathsf{P}$ and $\mathsf{\hat F}_n\triangleq\mathsf{C}_n\mathsf{B}(\mathcal{K}_{\sigma(n)})\mathsf{S}^N\mathsf{P}$. From the triangle inequality we get
$$
	\VERT\pi_n^\mu-\hat\pi_n^\mu\VERT_J \leq \VERT\pi_n^\mu-\tilde\pi_n^\mu\VERT_J + \VERT\tilde\pi_n^\mu-\hat\pi_n^\mu\VERT_J
$$
for some $J\subseteq{V}$ where
$$
	\tilde\pi_n^\mu = \mathsf{\tilde F_n}\cdots\mathsf{\tilde F}_1\mu
$$
with $\mathsf{\tilde F}_n\triangleq\mathsf{C}_n\mathsf{B}(\mathcal{K}_{\sigma(n)})\mathsf{P}$ is an ideal adaptively blocked filter; i.e. considering the adaptive blocking operation as it applies to the (non-sampled version of the) ideal Bayesian filter.

The expectation appearing in the definition of $\VERT \cdot \VERT_J$ is taken only with respect to the random sampling $\mathsf{S}^N$; see \cite{rebeschini2013can}. Hence, 
$$
	\VERT\pi_n^\mu-\tilde\pi_n^\mu\VERT_J = \sup_{f\in\mathfrak{X}^J: |f| \leq 1}\mathbf{E} \left[|\pi_n^\mu(f)-\tilde\pi_n^\mu(f)|^2 \right]^{1/2} = \sup_{f\in\mathfrak{X}^J: |f| \leq 1} |\pi_n^\mu(f)-\tilde\pi_n^\mu(f)|
$$
since no sampling occurs in $\pi_n^\mu$ or in $\tilde\pi_n^\mu$ and in this case $\VERT \cdot \VERT_J$ defines a local version of the total variation \cite{gibbs2002choosing} which, as in \cite{rebeschini2013can}, we sometimes denote by $\|\cdot\|_J$ for $J\subseteq{V}$. 

The first term in this decomposition quantifies the bias introduced by the blocking operation alone. The second term quantifies the 
error due to the variance of the random sampling. Typical analysis on the convergence of the bootstrap particle filter deals only with a variance term (since $\tilde\pi_n^\mu=\pi_n^\mu$ in that case).

Let $\Delta \triangleq \max_{v\in V}\mathrm{card}\{v'\in V:d(v,v')\le r\}$ and $\Delta_{\mathcal{K}} \triangleq \max_s\max_{K\in\mathcal{K}_s}\mathrm{card}\{K'\in\mathcal{K}_s:d(K,K')\le r\}$. Also define $\Delta_{d}(v) \triangleq \max_s d(v,\partial K_s)$, $\Delta_{d} \triangleq \max_v \Delta_d(v)$ and $\nabla_{d}(v) \triangleq \min_s d(v,\partial K_s)$. Finally, let $|\mathcal{K}|_\infty \triangleq \max_s\max_{K\in\mathcal{K}_s}|K|$.

\begin{theorem}[Bounding the Variance] \label{thm:variancebound}
There exists a constant $0<\varepsilon_0<1$, depending only on $\Delta$ and $\Delta_{\mathcal{K}}$ such that the following holds. Suppose there exist $\varepsilon_0<\varepsilon<1$ and $0<\kappa<1$ such that
$$
	\varepsilon \leq p^v(x,z^v) \leq \varepsilon^{-1},\quad \kappa \leq g^v(x^v,y^v) \leq \kappa^{-1}\quad \forall v\in V,~x,z\in\mathbb{X},~y\in\mathbb{Y}
$$
Then for every $n\geq 0$, $x\in\mathbb{X}$, and $v\in V$ we have
$$
	\VERT\tilde\pi_n^\mu-\hat\pi_n^\mu\VERT_v ~\leq~ \alpha \frac{|\mathcal{K}|_\infty e^{\beta|\mathcal{K}|_\infty}}{\sqrt{N}} 
$$
where $0<\alpha,\beta<\infty$ depend only on $\varepsilon$, $\kappa$, $r$, $\Delta$ and $\Delta_{\mathcal{K}}$.
\end{theorem}

The variance depends on the dimension of the sampling and so is necessarily dependent on the size of the blocks. This is exactly what is expected \cite{del2001stability,bickel2008sharp} of the variance in the sense that it recovers the behaviour of the standard bootstrap particle filter (without blocking) which is dependent on the size of the entire field.

The blocking operation essentially reduces the large-scale filtering problem to one of multiple, smaller, independent filtering problems and on which each independent particle filter mostly exhibits an error that is well understood \cite{del2001stability,bickel2008sharp}. Since the general nature (not the specific constants) of the variance bound is the best one might expect, we will not focus on this bound going forward. 

The main error component of relevant interest here is that component introduced purely as a result of the blocking operation (and not the sampling). Therefore, going forward we are largely concerned with $\tilde\pi_n^\mu = \mathsf{\tilde F_n}\cdots\mathsf{\tilde F}_1\mu$ where $\mathsf{\tilde F}_n\triangleq\mathsf{C}_n\mathsf{B}(\mathcal{K}_{\sigma(n)})\mathsf{P}$ and its ability to approximate the ideal, full, Bayesian filter $\pi_n^\mu = \mathsf{F_n}\cdots\mathsf{ F}_1\mu$. The particle filter $\hat\pi_n^\mu$ in this case can be thought of as an approximation of the ideal blocked filter $\tilde\pi_n^\mu$ and it is worth noting that other approximations to $\tilde\pi_n^\mu$ separate to particle-based approximations could be substituted. 

In summary, the main contribution is reduced to a study on adaptively blocked filtering and the error introduced through adaptive blocking when compared to the ideal Bayesian filter. The particle filtering step is given to show how one may approximate the adaptively blocked filter in practice and the error given on this particle representation is noted for completeness (this error is as expected even if it is non-trivial to derive).

The bound on the bias introduced due to blocking is now stated.

\begin{theorem}[Bounding the Bias] \label{thm:biasbound}
Suppose that $\varepsilon\leq p^v(x,z^v)\leq \varepsilon^{-1}$ for all $v\in V$ and $x,z\in\mathbb{X}$ with $\varepsilon >  \varepsilon_0 = \left(1-{1}/({18\Delta^2}) \right)^{1/2\Delta}$. Let $\beta = -(2r)^{-1}\log 18\Delta^2(1-\varepsilon^{2\Delta})>0$. If $~\sigma(s)=s~(\mathrm{mod}\,m)$, $s\in\mathbb{N}$ then for every $v\in V$ we have
$$
	\frac{1}{m}\sum_{k=0}^{m-1}\VERT\pi_{n-k}^x-\tilde\pi_{n-k}^x\VERT_v ~\leq~ \frac{8e^{-\beta}}{(1-e^{-\beta})}(1-\varepsilon^{2\Delta})\vartheta_m(v) ~\leq~ \frac{8e^{-\beta}}{(1-e^{-\beta})}(1-\varepsilon^{2\Delta})\exp\left[-\beta e^{-\beta(\Delta_d(v)-\nabla_d(v))} \frac{1}{m} \theta_m(v) \right]
$$
for every $n\geq 0$ and $x\in \mathbb{X}$ where $\theta_m(v) = \frac{1}{m}\sum_{j=0}^{m-1} d(v,\partial K_j(v))$ and $\vartheta_m(v) = \frac{1}{m}\sum_{j=0}^{m-1} e^{-\beta d(v,\partial K_j(v))}$ and $K_j(v)\in\mathcal{K}_j$. 
\end{theorem}

This is a time-uniform bound on the average bias over a time length of $m$. Both inequalities in Theorem \ref{thm:biasbound} imply that the bias introduced due to blocking can be spatially averaged (smoothed) across a cyclical application of a sequence of partitions. Both inequalities collapse to the result of \cite{rebeschini2013can} in the case $m=1$. The second inequality, in general, over bounds the first inequality but may be more convenient for discussion as $\theta_m(v)$ may be easier than $\vartheta_m(v)$ to conceptualise. Following the analysis in \cite{rebeschini2013can}, the goal was to derive a similarly natured bound here, but which captured honestly the spatial smoothing effect. 

The spatial invariance of the error bound (or more specifically the bias bound) will be discussed in more detail in the next section. We simply note here that if $\theta = \theta_m(v) = \frac{1}{m}\sum_{j=0}^{m-1} d(v,\partial K_j(v))$ where $K_j(v)\in\mathcal{K}_j$ for all $v\in\mathcal{V}$, then the bound really is spatially invariant. Such a situation occurs for a particular class of graphs (i.e. random fields) discussed later. The more general case in which $\theta \leq \frac{1}{m}\sum_{j=0}^{m-1} d(v,\partial K_j(v))$ is also discussed.

A simple corollary follows in which there exists an ordering of $\mathcal{K}_0, \dots, \mathcal{K}_{m-1}$ such that for a cyclical sequence of partitions $~\sigma(s)=s~(\mathrm{mod}\,m)$, $s\in\mathbb{N}$ we have
$$
	\VERT\pi_{n-k}^x-\tilde\pi_{n-k}^x\VERT_v ~\leq~ \tfrac{8e^{-\beta}}{(1-e^{-\beta})}(1-\varepsilon^{2\Delta})\vartheta_m(v) ~\leq~ \tfrac{8e^{-\beta}}{(1-e^{-\beta})}(1-\varepsilon^{2\Delta})\exp\left[-\beta e^{-\beta(\Delta_d(v)-\nabla_d(v))} \frac{1}{m} \theta_m(v) \right]
$$
for some (at least one) $v\in V$.

The blocking operation contributes the bias term to the total error while the random sampling contributes the variance term. As previously noted, the nature of the variance bound is as expected and we do not focus on that going forward. The bias is determined by the blocking operator which conceptually, at any time, has little effect on those sites far removed from the block borders due to the local dynamical dependencies assumed. Consequently, the bias is controlled at a sub-block level in that the bias at each site is controlled, on average, by that site's distance to the border of the blocks which contain it. If we can average this distance across the field through adaptive partitioning then we should be averaging the bias across the field.

From the computational view point, the variance bound implies that one need only consider the size of the blocks (not the entire random field) when picking a value for $N$ to control the error. This (ideally) leads to a reduction in the computational requirements of the filtering problem and is the underlying motivation for blocking. This gain comes at a price, in the form of a bias introduced due to blocking. Here we will consider adaptive blocking as a way of smoothing the bias error over the random field or controlling the bias in a more precise way.

Finally, we refer to \cite{rebeschini2013can} for a discussion on the mixing assumption $\varepsilon \leq p^v(x,z^v) \leq \varepsilon^{-1}$ with the non-standard requirement $\varepsilon_0<\varepsilon<1$ with $\varepsilon_0>0$. The restrictions on $\varepsilon_0<\varepsilon<1$ are relaxed partially in \cite{rebeschini2013comparison}. This term does not alter the significance of the result at the proof-of-concept stage. Moreover, non-optimal restrictions on the system model in this form appear frequently in similar studies; e.g. see \cite{del2001stability}. Typically, the empirical evidence suggests a far more relaxed application of the algorithms in question is permissible.

\subsection{Proof Strategy}

The bias and variance bounds are treated separately but lead to a total error bound. The proof strategy is adopted from Rebeschini and van Handel \cite{rebeschini2013can} and much of the analysis required is identical and not repeated. The strategy in \cite{rebeschini2013can} is inspired in part by the time-uniform particle filter convergence results \cite{del2013mean}. 

In the case of the bias $\VERT\pi_n^\mu-\tilde\pi_n^\mu\VERT_J$, one first derives a local stability property for the filter $\pi_n^\mu$ which implies that the marginal over a local set $J\subseteq V$ of the initial state $\mu$ is forgotten exponentially fast. Such a property also implies that any approximation errors in, say, the initial state are also forgotten. It then follows that if one can bound the one-step approximation error $\VERT\mathsf{F}_n\tilde\pi_{n-1}^\mu-\mathsf{\tilde F}_n\tilde\pi_{n-1}^\mu\VERT_J$ at any time, then in conjunction with the local stability property one will obtain a time-uniform bound on the bias over a local region of the field.

In the case of the variance $\VERT\tilde\pi_n^\mu-\hat\pi_n^\mu\VERT_J$, a similar idea is used except one first establishes stability for the ideal adaptively blocked filter $\tilde\pi_n^\mu$. Then, one must bound the one-step approximation error $\VERT\mathsf{\tilde F}_n\hat\pi_{n-1}^\mu-\mathsf{\hat F}_n\hat\pi_{n-1}^\mu\VERT_J$ at any time. Putting the stability property and the bound on the one-step approximation together, one achieves the desired time-uniform bound on the variance of a block in the adaptively blocked filter.

We have obviously glossed over much of the intricacies involved in the proof in this summary. For example, in the case of the bias, the property introduced in \cite{rebeschini2013can} and referred to as the decay of correlations must be established to hold uniformly in time for the ideal block filter $\tilde\pi_n^\mu$. This property captures a notion of spatial stability where the state at some site in the random field is forgotten as one moves away from that site. Rebeschini et al. provide a novel measure of this decay that allows them to establish local stability of the filter $\pi_n^\mu$ and to establish a bound on the one-step approximation error $\VERT\mathsf{F}_n\tilde\pi_{n-1}^\mu-\mathsf{\tilde F}_n\tilde\pi_{n-1}^\mu\VERT_J$. Conceptually, a property like the decay of correlations is necessary to establish such results. 

We refer the reader to \cite{rebeschini2013can} for a broader, more insightful, discussion on the strategy. Now, we note specifically what ideas must be altered to account for a change in partition from one time to the next. 

\subsubsection{Steps to Prove the Bias Bound}

Very roughly speaking the steps needed to prove the bound on the bias include: 1). Establishing the local stability property for the nonlinear filter $\pi_n^\mu$; and, 2). establishing the decay of correlations property holds uniformly in time for the ideal adaptively blocked filter $\tilde\pi_{n}^\mu$; and then, 3). establishing a bound on the one-step approximation error $\VERT\mathsf{F}_n\tilde\pi_{n-1}^\mu-\mathsf{\tilde F}_n\tilde\pi_{n-1}^\mu\VERT_J$ that holds at any time; and finally, 4) putting it all together.

The local stability of $\pi_n^\mu$ depends on the decay of correlations property assumed on $\mu$ and is otherwise independent of the blocking procedure. Hence, we can take this result as a given \cite{rebeschini2013can}. The one-step approximation error $\VERT\mathsf{F}_n\tilde\pi_{n-1}^\mu-\mathsf{\tilde F}_n\tilde\pi_{n-1}^\mu\VERT_J$ is dependent on the blocking procedure but only on the partition in effect during a single time step, and this partition is otherwise arbitrary, so we can take this result as given \cite{rebeschini2013can}.

To prove our case, we only need to establish that the decay of correlations property holds uniformly in time for the filter distribution $\tilde\pi_{n}^\mu$ when given the changing partitions. Once this is established, it is just a matter of collecting the relevant results and finalising the bound on $\VERT\pi_n^\mu-\tilde\pi_n^\mu\VERT_J$. We follow through with this last step and show how the spatial averaging effect of $\theta_m(v) = \frac{1}{m}\sum_{j=0}^m d(v,\partial K_j(v))$ comes out during this procedure.

The detailed proof is given in a subsequent section drawing from \cite{rebeschini2013can} as often as possible.

\subsubsection{Steps to Prove the Variance Bound}

Roughly again, the steps needed to prove the bound on the variance include: 1). Establishing the stability of the ideal adaptively blocked filter $\tilde\pi_{n}^\mu$; and, 2). establishing a bound on the one-step approximation error $\VERT \tilde{\mathsf{F}}_n\hat\pi_{n-1}^\mu-\mathsf{\hat F}_n\hat\pi_{n-1}^\mu \VERT_J$ that holds at any time; and finally, 3). putting it all together.

Firstly, we do not have to deal with any correlation-like properties in the case of the variance bound and as noted the final result is as expected. So things may appear simpler initially. Unfortunately, proving stability for the ideal adaptively blocked filter $\tilde\pi_{n}^\mu$ is not trivial \cite{rebeschini2013can}. Because the stability of $\tilde\pi_{n}^\mu$ is dependent on the change of partition we must re-establish that this stability result holds in the case of adaptively changing partitions for completeness. Moreover, for technical reasons related to the use of the norm $\VERT \cdot \VERT_J$, the authors in \cite{rebeschini2013can} consider instead a two-step approximation error $\VERT \tilde{\mathsf{F}}_{n+1}\tilde{\mathsf{F}}_n\hat\pi_{n-1}^\mu-\tilde{\mathsf{F}}_{n+1}\mathsf{\hat F}_n\hat\pi_{n-1}^\mu \VERT_J$ and bound this term at any time. Because a change in partition comes into play over two steps, we must re-establish that this two-step approximation error is bounded under a change of partition at any time. We then bring the relevant results together and finalise the bound on $\VERT\tilde\pi_n^\mu-\hat\pi_n^\mu\VERT_J$.

As noted, this is the strategy taken to derive the variance bound but, since our main concern here is the spatial aspects of the filtering problem and the related (adaptive) blocking operation, we do not give the details. Many of the technical lemmas involved in this analysis also follow directly from \cite{rebeschini2013can} and those results requiring a modification to their proofs need only an arguably minor re-analysis and modification. The final result is as expected and the authors are available to provide the variance bound proof details on request.

\section{Discussion on the Adaptively Blocked Filter and Spatial Smoothing}

The main result in the previous section is a total error bound on $\VERT\pi_n^\mu-\hat\pi_n^\mu\VERT_v$, for all $v\in V$ and can be decomposed (and is actually derived) in terms of a bound on the variance (induced by the random Monte Carlo procedure) and a bound on the bias (induced by the blocking operator). We focus on the bias bound in this section and its relevance as it pertains to the dependence of the total error on the particular spatial site $v$. 

The main point of interest in this work is the effect of the (adaptive) blocking operation on the total error bound which shows up purely via the systematic bias. To this end, we compare the bound on the bias proposed here with the bias bound proposed in the motivating paper \cite{rebeschini2013can} by Rebeschini et al.,
\begin{align*}
	\mathrm{bias}_{\mathrm{B}}(v) &\leq \mathcal{O}(\tfrac{1}{m}\textstyle\sum_{j=0}^{m-1} e^{- d(v,\partial K_j(v))})  = \mathcal{O}(\vartheta_m(v)) \leq\mathcal{O}(e^{\frac{-1}{m}\sum_{j=0}^{m-1} d(v,\partial K_j(v))}) = \mathcal{O}(e^{-\theta_m(v)})  &\mathrm{Bertoli~et~al.} \\
	\mathrm{bias}_{\mathrm{R}}(v) &\leq \mathcal{O}(e^{-d(v,\partial K(v))}) & \mathrm{Rebeschini~et~al.}
\end{align*}
where Rebeschini et al. only ever consider a single partition. Here, $\mathrm{bias}_{\mathrm{B}}(v)$ and $\mathrm{bias}_{\mathrm{R}}(v)$ can be taken as the average of the bias over a time period of length $m$. We remove any unnecessary constants from the expressions that cloud the conceptual discussion. Here, we use $\mathcal{O}(\cdot)$ to capture only that spatially dependent component of the bias bound noting that the constant `out-the-front' is equivalent in both cases \cite{rebeschini2013can}. For conceptual, rather than technical, reasons we consider the slightly looser bound $\mathrm{bias}_{\mathrm{B}}(v) \leq \mathcal{O}(e^{-\theta_m(v)})$ during discussion. 

We highlight that if $\theta_m(v) = \frac{1}{m}\sum_{j=0}^{m-1} d(v,\partial K_j(v))$ for all $v\in\mathcal{V}$ then
$$
	\mathcal{O}(e^{\frac{-1}{m}\sum_{j=0}^{m-1} d(v,\partial K_j(v))})=\mathcal{O}(e^{-\theta_m(v)})=\mathcal{O}(e^{-\theta})
$$ 
and the bound is truly spatially invariant. This is part of the motivation for this work and is explored in more detail now. Consider Figure \ref{fig:graph1}.

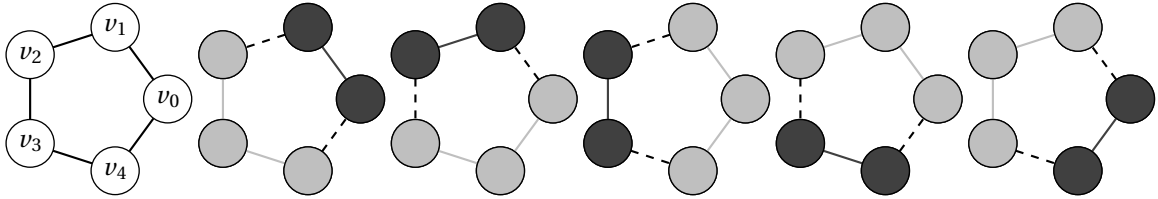
\begin{figure}[!h]
\centering
\begin{tikzpicture}
	\grCirculant[RA=1, Math=true, prefix=v]{5}{1}
\end{tikzpicture}
\begin{tikzpicture}
	\grEmptyCycle[RA=1, Math=true, prefix=v]{5}
	\Edges[local,color=darkgray](v0,v1)
	\Edges[local,color=lightgray](v2,v3,v4)
	\Edges[local,style=dashed](v1,v2)
	\Edges[local,style=dashed](v4,v0)
	\AddVertexColor{darkgray}{v0,v1}
	\AddVertexColor{lightgray}{v2,v3,v4}
\end{tikzpicture}
\begin{tikzpicture}
	\grEmptyCycle[RA=1, Math=true, prefix=v]{5}
	\Edges[local,color=darkgray](v1,v2)
	\Edges[local,color=lightgray](v3,v4,v0)
	\Edges[local,style=dashed](v2,v3)
	\Edges[local,style=dashed](v0,v1)
	\AddVertexColor{darkgray}{v1,v2}
	\AddVertexColor{lightgray}{v3,v4,v0}
\end{tikzpicture}
\begin{tikzpicture}
	\grEmptyCycle[RA=1, Math=true, prefix=v]{5}
	\Edges[local,color=darkgray](v2,v3)
	\Edges[local,color=lightgray](v4,v0,v1)
	\Edges[local,style=dashed](v3,v4)
	\Edges[local,style=dashed](v1,v2)
	\AddVertexColor{darkgray}{v2,v3}
	\AddVertexColor{lightgray}{v4,v0,v1}
\end{tikzpicture}
\begin{tikzpicture}
	\grEmptyCycle[RA=1, Math=true, prefix=v]{5}
	\Edges[local,color=darkgray](v3,v4)
	\Edges[local,color=lightgray](v0,v1,v2)
	\Edges[local,style=dashed](v4,v0)
	\Edges[local,style=dashed](v2,v3)
	\AddVertexColor{darkgray}{v3,v4}
	\AddVertexColor{lightgray}{v0,v1,v2}
\end{tikzpicture}
\begin{tikzpicture}
	\grEmptyCycle[RA=1, Math=true, prefix=v]{5}
	\Edges[local,color=darkgray](v4,v0)
	\Edges[local,color=lightgray](v1,v2,v3)
	\Edges[local,style=dashed](v0,v1)
	\Edges[local,style=dashed](v3,v4)
	\AddVertexColor{darkgray}{v4,v0}
	\AddVertexColor{lightgray}{v1,v2,v3}
\end{tikzpicture}
\caption{The left most circulant graph depicts the underlying graphical structure of the random field. The remaining five graphs from second left to far right highlight a sequence of independent partitions of the original graph in which each partition consists of a set of two (dark grey) and three (light grey) vertices.} \label{fig:graph1}
\end{figure}

Pick any site $v\in V$ in the graph depicted in Figure \ref{fig:graph1} and note that $\theta = \theta_m(v) = 1/5$. One then clearly has a spatially uniform bound on $\mathrm{bias}_{\mathrm{B}}(v)$. Consider now any single partition alone and note that for four out of the five sites we have $d(v,\partial K(v))=0$ and at one site we have $d(v,\partial K(v))=1$ which implies, as noted by Rebeschini et al., that the bound on $\mathrm{bias}_{\mathrm{R}}(v)$ is not spatially uniform. The adaptive blocking procedure is averaging the distance $d(v,\partial K_j(v))$ through the use of multiple partitions which results in a kind of spatial error smoothing.

The bound on the bias of the cyclically blocked filter at every site $v\in V$ is completely independent of the site $v$ in every case in which $\theta = \theta_m(v)$, $\forall v\in V$. Such cases may occur in practice; e.g. a sequence of partitions on any regular lattice wrapped on a torus can be derived that obeys this property, see Figure \ref{fig:graph2}.

\begin{figure}[!h]
\begin{center}
\vspace{-0.5cm}
\begin{tikzpicture}
\selectcolormodel{gray}
\begin{axis}[view={80}{60}, axis lines=none]
	\addplot3[mesh, scatter, fill=white, opacity=0.5, z buffer=sort, domain=0:2*pi, y domain=0:2*pi]
		({(4+cos(deg(x)))*cos(deg(y))}, {(4+cos(deg(x)))*sin(deg(y))}, {sin(deg(x))});
\end{axis}
\end{tikzpicture}
\end{center}
\vspace{-1.2cm}
\caption{A lattice wrapped on a torus with no physical boundaries.} \label{fig:graph2}
\end{figure}
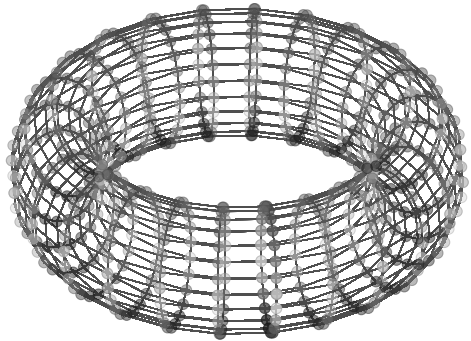

In the general case, in which $\theta \leq \theta_m(v)$ for some $v\in V$, the spatial smoothing property of the cyclical blocking filter is still in effect and reduces, as compared to \cite{rebeschini2013can}, the degree of spatial inhomogeneity (on average) as it applies to the bias bound. Essentially, the sites on the borders of a block in $\mathcal{K}$ are typically not on the borders of a block in $\mathcal{K}'$, while the sites at the centre of a block in $\mathcal{K}$ are typically not at the centre of a block in $\mathcal{K}'$. Given a sufficient number of well-chosen partitions of this type, then one can ensure the average distance of a site to a border is smoothed (or spatially averaged) across all sites. Of course, this means that some particular sites may be worse off than they were under a single partition (this is an obvious consequence averaging). 

For example, consider again the case in Figure \ref{fig:graph1} but suppose only the four left most partitions are employed by the cyclically blocked particle filter. Then $\theta_m(v_0)=\theta_m(v_1)=\theta_m(v_3)=\theta_m(v_4)=1/4$ while $\theta_m(v_2)=0$. Clearly, one has a more desirable bound on the bias on average in this case than in the case in which only a single partition is considered, albeit complete spatial homogeneity is not achieved. Considering additional partitions in a large-scale random field will be of even further benefit than that exposited in this toy example.

Note finally that in both Rebeschini et al. and here the spatial uniformity of the error (or the bias more specifically) is often referred to via the bound on the bias and not the bias or error itself. This is a consequence of the technical analysis, but for all practical purposes it would appear obvious that the spatial homogeneity of the error itself is of the same nature as that noted by the bound applicable to that error (even if such bounds are otherwise quite loose). That is, the blocked filter of Rebeschini et al. \cite{rebeschini2013can} would clearly seem to favour those sites far from the border of the individual blocks in terms of the actual performance of the filter, while the cyclically blocked filter proposed in this work is clearly, in some sense, averaging out this favouritism and its effect on the actual filter performance at any site. The point is that the spatial relationship of the error is typically noted in terms of the bias bounds but intuitively/conceptually the discussions on homogeneity (or inhomogeneity) of a particular filter apply (seemingly) also to the error/bias itself.

\section{Proof of the Bias Bound}

Rebeschini et al. \cite{rebeschini2013can} introduced an important concept referred to as the decay of correlation which captures, in a very technical manner, the intuitive notion of spatial stability where the state at some site in the random field should be forgotten as one moves away from that site. This notion plays a crucial role in the convergence of the bias due to the blocking operation.

It is important to note that the analysis and the spatial stability property put forth in \cite{rebeschini2013can} is based in part on those ideas of temporal stability introduced in \cite{del2001stability} and used to establish time-uniform convergence results for the standard bootstrap particle filter.

Recall that the dynamics of the underlying process $(X_n)_{n\geq 0}$ are local in the sense that $p^v(x,z^v)$ depends only on $x^{N(v)}$ where $x^J=(x^j)_{j\in J}$ for $J\subseteq V$. Here, $N(v)=\{v'\in V:d(v,v')\le r\}$ for some $r\in\mathbb{N}$ captures the local neighbourhood of sites on which site $v$ explicitly depends. 

We now briefly review the measure introduced in \cite{rebeschini2013can} on the decay of correlations property. For any probability measure $\mu$ on $\mathbb{X}$ and for $x,z\in\mathbb{X}$ with $v\in V$ define
$$	
	\mu_{x,z}^v(A) ~\triangleq ~\mathbf{P}^\mu[X_{n-1}^v \in A | X_{n-1}^{V\setminus\{v\}}=x^{V\setminus\{v\}},~X_n=z]  = ~\frac{\int \mathbf{1}_A(x^v)\prod_{u\in N(v)}p^u(x,z^u) \mu^v_x(dx^v)}{\int \prod_{u\in N(v)}p^u(x,z^u) \mu^v_x(dx^v)}
$$
for any $A\in\mathcal{X}$ and where
$$	
	\mu^v_x(dx^v) ~\triangleq~ \mu(X_{n-1}^v = dx^v | X_{n-1}^{V\setminus\{v\}}=x^{V\setminus\{v\}})
$$
Then
$$
	C^\mu_{vv'} ~\triangleq~ \frac{1}{2} \sup_{z\in\mathbb{X}}~~\sup_{x,\tilde x\in\mathbb{X}~:~x^{V\setminus\{v'\}}=\tilde x^{V\setminus\{v'\}}}
				~ \|\mu^v_{x,z}-\mu^v_{\tilde x,z}\|
$$
for $v,v'\in V$. This quantity $C^\mu_{vv'}$ somehow captures the correlation between two sites $v,v'\in V$ in the random field under the assumed field model. A little more precisely, this term is measuring the maximal total variation at a site $v$ that may arise due to a perturbation at site $v'$. Now define~
$$
	\corr(\mu,\beta) ~\triangleq~ \max_{v\in V} \sum_{v'\in V} e^{\beta d(v,v')} C^\mu_{vv'}
$$
with $\beta>0$. This quantity $\corr(\mu,\beta)$ is a measure on the total degree of correlation decay for the measure $\mu$ given a rate parameter $\beta$. The site $v$ can be interpreted as the most sensitive site in the field. To understand this quantity $\corr(\mu,\beta)$ a little more conceptually, suppose the most sensitive site $v$ is known a priori. Then suppose that $\sum_{v'\in V} e^{\beta d(v,v')} C^\mu_{vv'}=1$. It follows that the correlation between any two sites in the random field decays as a function of the distance between those two sites at an exponential rate defined by at least $\beta$.

With $\corr(\mu,\beta)$ defined as such we borrow directly from \cite{rebeschini2013can} the local stability result on $\pi_n^\mu$ which requires only that the initial condition $\mu$ satisfy a decay of correlations property. Recall that $\Delta \triangleq \max_{v\in V}\mathrm{card}\{v'\in V:d(v,v')\leq r\}$ defines the size of the largest neighbourhood in $V$.

\begin{lemma}[Local Filter Stability \cite{rebeschini2013can}] \label{lem:localfilterstability}
Suppose there exists $\varepsilon>0$ such that $\varepsilon \leq p^v(x,z^v)\leq \varepsilon^{-1}$ for all $v\in V$ and $x,z\in\mathbb{X}$. Let $\mu,\nu$ be probability measures on $\mathbb{X}$, and suppose that
$$
	\corr(\mu,\beta) \leq \frac{1}{2} - 3(1-\varepsilon^{2\Delta})e^{2\beta r}\Delta^2
$$
for a sufficiently small constant $\beta>0$. Then
$$
	\VERT \mathsf{F}_n\cdots\mathsf{F}_{s+1}\mu - \mathsf{F}_n\cdots\mathsf{F}_{s+1}\nu \VERT_J \leq 2e^{-\beta(n-s)} \sum_{v\in J}\max_{v'\in V} e^{-\beta d(v,v')} \sup_{x,z\in\mathrm{X}} \VERT\mu^{v'}_{x,z}-\nu^{v'}_{x,z}\VERT
$$
for every $J\subseteq V$ and $s<n$.
\end{lemma}

Now consider any partition $\mathcal{K}_i$. We define a correlation depending on the given partition. Fix a probability measure $\mu$ on $\mathbb{X}$ and $x,z \in \mathbb{X}$, $v \in V$, $K \in \mathcal{K}_i$ and then let
$$
	\mu^{v, K}_{x,z}(A) ~\triangleq~ \mathbf{P}^{\mu}[ X_{n-1}^v \in A | X_{n-1}^{V\setminus \{v\}} = x^{V\setminus \{v\}},  X_n^{K}=z^{K}] = \frac{\int \mathbf{1}_A(x^v) \prod_{u\in N(v)\cap K} p^u(x,z^u)\mu^v_x(dx^v)}{\int \prod_{u\in N(v)\cap K} p^u(x,z^u)\mu^v_x(dx^v)}
$$
for any $A\in\mathcal{X}$. Now define
$$
	\widetilde{C}^{\mathcal{K}_i,\mu}_{vv'} ~\triangleq~ \frac{1}{2} ~\max_{K \in \mathcal{K}_i}~~ \sup_{z\in \mathbb{X}}  ~~ \sup_{x,\widetilde{x}\in \mathbb{X}~:~ x^{V\setminus \{v'\}} = \widetilde{x}^{V\setminus \{v'\}}} ~ \| \mu ^{v, K}_{x,z} - \mu ^{v, K}_{\widetilde{x},z}\|
$$
and
$$
	\widetilde{\corr}_{\mathcal{K}_i}(\mu,\beta) ~\triangleq~ \max_{v \in V}\sum_{v'\in V}e^{\beta d(v,v')}\widetilde{C}^{\mathcal{K}_i,\mu}_{vv'}
$$
for any partition $\mathcal{K}_i$ in the cyclically blocked particle filter sequence. One can interpret the block adapted measure $\widetilde{\corr}_{\mathcal{K}_i}(\mu,\beta)$ in much the same way as $\corr(\mu,\beta)$. 

The reason behind formulating the measure $\corr(\mu,\beta)$ in such a way is technical and follows from the program put forth in \cite{rebeschini2013can}. Here we are only modifying this program to account for changing partitions and in doing so we seek to draw on the detailed analysis put forth in \cite{rebeschini2013can} as much as possible. The reason for introducing $\widetilde{\corr}_{\mathcal{K}_i}(\mu,\beta)$ is now explained. In order to establish the one-step approximation error we must bound $\corr(\tilde\pi_{n}^\mu,\beta)$ uniformly in time. The authors of \cite{rebeschini2013can} note the difficulty in working directly with $\corr(\tilde\pi_{n}^\mu,\beta)$ and instead propose to bound $\widetilde{\corr}_{\mathcal{K}_{\sigma(n)}}(\tilde\pi_{n}^\mu,\beta)$ and then use the following result to indirectly control $\corr(\tilde\pi_{n}^\mu,\beta)$.

\begin{lemma}[\cite{rebeschini2013can}] \label{lem:tildecorrboundcorr}
For any probability measure $\nu$, partition $\mathcal{K}_i$ and $\beta>0$, we have
$$
	\corr(\nu,\beta)\leq (1-\varepsilon^{2\Delta})e^{2\beta r}\Delta^2 + 2\varepsilon^{-2\Delta} \widetilde{\corr}_{\mathcal{K}_i}(\nu,\beta)
$$
\end{lemma}

Given this result it follows that a time-uniform bound on $\widetilde{\corr}_{\mathcal{K}_{\sigma(n)}}(\tilde\pi_{n}^\mu,\beta)$ leads easily to a time-uniform bound on $\corr(\tilde\pi_{n}^\mu,\beta)$. 

Because $\tilde\pi_{n}^\mu$ is dependent on the changing partitions, which is central to the adaptively blocked filter, it follows that we must re-establish that a time-uniform bound on $\widetilde{\corr}_{\mathcal{K}_{\sigma(n)}}(\tilde\pi_{n}^\mu,\beta)$ exists in the case of interest here. 

We prove a technical lemma that will be used subsequently.

\begin{lemma} \label{lem:minor}
For any strictly positive $a,b\in\mathbb{R}_+$ and $x\in\mathbb{R}_+$ with $x\geq1$ the following holds: $|ax-b/x| \leq |a-b|x^2$.
\end{lemma}
\begin{proof}
If $a\geq b$ or even $ax\geq b/x$ then
$$
	|ax-\frac{b}{x}| ~=~ ax-bx+bx-\frac{b}{x} ~=~ |a-b|x + b\frac{x^2-1}{x} ~\leq~ |a-b|x ~\leq~ |a-b|x^2
$$
Note $ax\geq b/x~$ implies $x\geq\sqrt{b/a}$ for any $a,b,x>0$. This leaves the case $b> a$ and $1\leq x <\sqrt{b/a}$. In this case
$$
	|ax-\frac{b}{x}| ~=~ |ax^2-b|\frac{1}{x} ~\leq~ |ax^2-b| ~\leq |a-b| \leq |a-b|x^2
$$
by continuity in $1\leq x <\sqrt{b/a}~$ for any fixed $b> a>0$.
\end{proof}

The following proposition relates $\widetilde{\corr}_{\mathcal{K}_i}(\mu,\beta)$ under one partition to the same quantity considered under a different partition and is needed to obtain the time-uniform bound on the $\widetilde{\corr}_{\mathcal{K}_{\sigma(n)}}(\tilde\pi_{n}^\mu,\beta)$. 

\begin{proposition} \label{prop:tildecorrboundchangepartition}
Suppose there exists $\varepsilon>0$ such that $\varepsilon \leq p^v(x,z^v)\leq \varepsilon^{-1}$. For any probability measure $\mu$, partitions $\mathcal{K}_i, \mathcal{K}_j$ and $\beta>0$, we have
$$
	\widetilde{\corr}_{\mathcal{K}_i}(\mu,\beta) \leq \varepsilon ^{-8\Delta} \widetilde{\corr}_{\mathcal{K}_j}(\mu,\beta)
$$
\end{proposition}
\begin{proof}
Pick $K_i \in \mathcal{K}_i$ and $K_j \in \mathcal{K}_j$ with $K \triangleq K_i \cap K_j \neq \emptyset$. Then by definition
\begin{align*}
	\mu ^{v, K_i}_{x,z}(A) & = \frac{\int \mathbf{1}_A(x^v) \prod_{u\in N(v)\cap K_i\backslash K} p^u(x,z^u)\mu^{v,K}_{x,z}(dx^v)}{\int \prod_{u\in N(v)\cap K_i\backslash K} p^u(x,z^u)\mu^{v,K}_{x,z}(dx^v)} 
	 \leq\frac{\int \mathbf{1}_A(x^v)\varepsilon^{-|N(V)\cap K_i\backslash K|} \prod_{u\in N(v)\cap K_j\backslash K} \frac{p^u(x,z^u)}{\varepsilon} \mu^{v,K}_{x,z}(dx^v)}{\int \varepsilon^{|N(V)\cap K_i\backslash K|} \prod_{u\in N(v)\cap K_j\backslash K} \frac{p^u(x,z^u)}{\varepsilon^{-1}}\mu^{v,K}_{x,z}(dx^v)} \\
	& \leq \varepsilon^ { -2 [|N(V)\cap K_i\backslash K| + |N(V)\cap K_j\backslash K|]} \mu ^{v, K_j}_{x,z}(A) \\
	& \leq \varepsilon^ { -4\Delta}\mu ^{v, K_j}_{x,z}(A)
\end{align*}
for $A\in\mathcal{X}$. Alternatively,
$$
\mu ^{v, K_i}_{x,z}(A)\geq \varepsilon^ { 4\Delta}\mu ^{v, K_j}_{x,z}(A)
$$
noting that $\mathcal{K}_i$ and $\mathcal{K}_j$ are anyway arbitrary. Fix $A\in\mathcal{X}$ and $x,\tilde x\in\mathbb{X}$ and suppose, without loss of generality, $\mu ^{v, K_i}_{x,z}(A)\geq\mu ^{v, K_i}_{\widetilde{x},z}(A)$. It follows
$$
|\mu ^{v, K_1}_{x,z}(A)- \mu ^{v, K_1}_{\widetilde{x},z}(A)|=\mu ^{v, K_1}_{x,z}(A)- \mu ^{v, K_1}_{\widetilde{x},z}(A) \leq \varepsilon^ { -4\Delta} \mu ^{v, K_2}_{x,z}(A)- \varepsilon^ { 4\Delta}\mu ^{v, K_2}_{\widetilde{x},z}(A) \leq \varepsilon^ {- 8\Delta}|\mu ^{v, K_2}_{x,z}(A)-\mu ^{v, K_2}_{\widetilde{x},z}(A)|
$$
using Lemma \ref{lem:minor}. Taking the supremum over $A\in\mathcal{X}$ gives
$$
\|\mu ^{v, K_1}_{x,z}- \mu ^{v, K_1}_{\widetilde{x},z}\|\leq \varepsilon^ {- 8\Delta}\|\mu ^{v, K_2}_{x,z}-\mu ^{v, K_2}_{\widetilde{x},z}\|
$$
and from this the result of the proposition follow easily.
\end{proof}

We state the following lemma which bounds the change in the correlation decay over a one-step application of the ideal cyclically blocked filter for a fixed partition.

\begin{lemma}[\cite{rebeschini2013can}] \label{lem:tildecorrboundonestep}
Suppose there exists $\varepsilon>0$ such that $\varepsilon\le p^v(x,z^v)\leq \varepsilon^{-1}$ for all $v\in V$ and $x,z\in\mathbb{X}$. For any probability measure $\nu$, partition $\mathcal{K}_{\sigma(s)}$ and sufficiently small $\beta>0$ such that
$$
	\widetilde{\corr}_{\mathcal{K}_{\sigma(s)}}(\nu,\beta)  ~\leq~ {1}/{2} - (1-\varepsilon^2)e^{\beta(r+1)}\Delta
$$
then
$$
	\widetilde{\corr}_{\mathcal{K}_{\sigma(s)}}(\mathsf{\tilde F}_s\nu,\beta) ~\leq~ 2(1-\varepsilon^{2\Delta})e^{2\beta r}\Delta^2
$$
for any $s\in\mathbb{N}$.
\end{lemma}

Now we are in a position to prove the time-uniform bound on $\widetilde{\corr}_{\mathcal{K}_{\sigma(n)}}(\tilde\pi_{n}^\mu,\beta)$. We have via Lemma \ref{lem:tildecorrboundonestep} a bound on the change in the correlation decay over any single time step. We have via Proposition \ref{prop:tildecorrboundchangepartition} a relationship between the decay of correlations for a measure under two different partitions. We combine these results and iterate to get the desired time-uniform bound as now shown.

\begin{lemma} \label{lem:tildecorrbounduniformtime}
Assume $\varepsilon\le p^v(x,z^v)\leq \varepsilon^{-1}$ for all $v\in V$ and $x,z\in\mathbb{X}$ with
$$
	\varepsilon>\varepsilon_0= \left (1-{1}/(16\Delta^2) \right ) ^{\frac{1}{2\Delta}}
$$
Let $\mu$ be a probability measure on $\mathbb{X}$ and $\mathcal{K}_0$ a partition of $V$ such that
$$
	\widetilde{\corr}_{\mathcal{K}_{\sigma(0)}}(\mu,\beta) ~\leq~ {1}/{8}
$$
where $\beta=-\frac{1}{2r}\log[16 \Delta^2(1-\varepsilon^{2\Delta})]>0$. Then
$$
	\widetilde{\corr}_{\mathcal{K}_{\sigma(n)}}(\tilde{\pi}^{\mu}_n, \beta) ~\leq~ {1}/{8} 
$$
for all $n \geq 0$.
\end{lemma}

\begin{proof}
First, $\varepsilon>\varepsilon_0= \left (1- {1}/(16\Delta^2) \right)^{\frac{1}{2\Delta}}$ implies $\varepsilon^{-8\Delta}\leq 2$ and $(1-\varepsilon^2)e^{\beta(r+1)}\Delta \leq {1}/{16}$. Thus, given $\widetilde{\corr}_{\mathcal{K}_{\sigma(0)}}(\mu,\beta) \leq {1}/{8}$ we have
$$
	\widetilde{\corr}_{\mathcal{K}_{\sigma(1)}}(\mu,\beta) ~\leq~ 2\,\widetilde{\corr}_{\mathcal{K}_{\sigma(0)}}(\mu,\beta) \leq \frac{1}{4}
$$
via Proposition \ref{prop:tildecorrboundchangepartition}. Now, given $\widetilde{\corr}_{\mathcal{K}_{\sigma(1)}}(\mu,\beta) \leq 1/4 < 7/16 \leq {1}/{2} - (1-\varepsilon^2)e^{\beta(r+1)}$ we have
$$
	\widetilde{\corr}_{\mathcal{K}_{\sigma(1)}}(\tilde{\pi}^{\mu}_1,\beta) ~=~ \widetilde{\corr}_{\mathcal{K}_{\sigma(1)}}(\mathsf{\tilde F}_1\mu,\beta) ~\leq~ 2(1-\varepsilon^{2\Delta})e^{2\beta r}\Delta^2 ~\leq~ \frac{1}{8}
$$
via Lemma \ref{lem:tildecorrboundonestep}. Just restart the argument and iterate to get $\widetilde{\corr}_{\mathcal{K}_{\sigma(n)}}(\tilde{\pi}^{\mu}_n, \beta)\leq {1}/{8}$ for any $n\geq0$.
\end{proof}

Now we have a time-uniform bound on $\widetilde{\corr}_{\mathcal{K}_{\sigma(n)}}(\tilde{\pi}^{\mu}_n, \beta)$, and we can use Lemma \ref{lem:tildecorrboundcorr} to arrive at the result we want which is a time-uniform bound on the decay of correlation measure of interest.

\begin{lemma}[Bounding the Decay of Correlations] \label{lemma:corrbounduniformtime}
Assume $\varepsilon\le p^v(x,z^v)\leq \varepsilon^{-1}$ for all $v\in V$ and $x,z\in\mathbb{X}$ with
$$
	\varepsilon>\varepsilon_0= \left (1-{1}/(16\Delta^2) \right ) ^{\frac{1}{2\Delta}}
$$
Let $\beta = -(2r)^{-1}\log 16\Delta^2(1-\varepsilon^{2\Delta})>0$. Then
$$
	\corr(\tilde\pi_n^\mu,\beta) \leq \frac{1}{3}
$$
for every $n\geq 0$ and any probability measure $\mu$ on $\mathbb{X}$.
\end{lemma}

Recall the program required to prove the bias bound: 1). Establish the local stability property for the nonlinear filter $\pi_n^\mu$; and, 2). establish the decay of correlations property holds uniformly in time for the ideal cyclical blocked filter $\tilde\pi_{n}^\mu$; and then, 3). establish a bound on the one-step approximation error $\VERT\mathsf{F}_n\tilde\pi_{n-1}^\mu-\mathsf{\tilde F}_n\tilde\pi_{n-1}^\mu\VERT_J$ that holds at any time; and finally, 4). put it all together.

All that remains is to show that $\VERT\mathsf{F}_n\tilde\pi_{n-1}^\mu-\mathsf{\tilde F}_n\tilde\pi_{n-1}^\mu\VERT_J$ is bounded at any time and then to put it all together. Since this one-step approximation error is independent of any change in partition we refer back to \cite{rebeschini2013can}.

\begin{lemma}[Bounding the One-Step Blocking Approximation Error \cite{rebeschini2013can}] \label{lem:biasboundonestep}
Suppose there exists $\varepsilon>0$ such that $\varepsilon\leq p^v(x,z^v)\leq\varepsilon^{-1}$ for all $v\in V$ and $x,z\in\mathbb{X}$. Let $\nu$ be a probability measure on $\mathbb{X}$, and suppose that
$$
	\corr(\nu,\beta) \leq \tfrac{1}{2} - (1-\varepsilon^2)e^{\beta(r+1)}\Delta
$$
for a sufficiently small constant $\beta>0$. Then
$$
	\sup_{x,z\in\mathbb{X}} \VERT (\mathsf{F}_s\nu)^v_{x,z}-(\mathsf{\tilde F}_s\nu)^v_{x,z} \VERT \leq 4e^{-\beta}(1-\varepsilon^{2\Delta})e^{-\beta d(v,\partial K(v))}
$$
for every $s\in\mathbb{N}$ and every $K\in\mathcal{K}_{\sigma(s)}$ with $v\in K$. Furthermore,
$$
	\VERT\mathsf{F}_n\nu-\mathsf{\tilde F}_n\nu\VERT_{J} \leq 4|J| e^{-\beta}(1-\varepsilon^{2\Delta}) e^{-\beta d(J,\partial K)}  
$$
for every $K\in\mathcal{K}_{\sigma(n)}$ and $J\subseteq K$.
\end{lemma}

Now we are ready to put all the results together and finalise the bound on the bias.

\subsection{Proof of Theorem \ref{thm:biasbound}}

Note $\tilde\pi_0^\mu=\mu$ and here we consider the case $\mu=\delta(x)$ for some $x\in\mathbb{X}$. Firstly, note
$$
	\varepsilon >  \varepsilon_0 = \left(1- {1}/(18\Delta^2) \right)^{1/2\Delta} > \left(1-{1}/(16\Delta^2) \right)^{1/2\Delta}
$$
and $\beta = -(2r)^{-1}\log 16\Delta^2(1-\varepsilon^{2\Delta})>0$ such that Lemma \ref{lemma:corrbounduniformtime} holds. Moreover, 
$$
	\corr(\tilde\pi_{n}^x,\beta) ~\leq~ 1/3 ~\leq~ \tfrac{1}{2} - (1-\varepsilon^2)e^{\beta(r+1)}\Delta \geq 4/9
$$
for every $n\geq 0$ and Lemma \ref{lem:biasboundonestep} holds. Finally, 
$$
	\corr(\tilde\pi_n^x,\beta) ~\leq~ 1/3 \leq 1/2 - 3(1-\varepsilon^{2\Delta})e^{2\beta r}\Delta^2 \geq 1/3
$$
for every $n\geq 0$ and Lemma \ref{lem:localfilterstability} holds. Thus, the bound on the decay of correlations holds and implies local filter stability and the one-step blocking approximation error bound all hold under the given parameter hypotheses.

Fix $K_{\sigma(s)} \in \mathcal{K}_{\sigma(s)}$ so $J\subseteq K_{\sigma(s)}$ for all $s\in\mathbb{N}$. Then
\begin{align*}
	\VERT \pi^{x}_n - \widetilde{\pi}^{x}_n \VERT_J & \leq \sum_{s=1}^n \VERT \mathsf{F}_n\dots\mathsf{F}_{s+1}\mathsf{F}_s\tilde{\pi}^{x}_{s-1} - \mathsf{F}_n\dots\mathsf{F}_{s+1}\tilde{\mathsf{F}}_s\tilde{\pi}^{x}_{s-1} \VERT_J \\
	& \leq \VERT \mathsf{F}_n\tilde{\pi}^{x}_{n-1} - \tilde{\mathsf{F}}_n\tilde{\pi}^{x}_{n-1} \VERT_J + \sum_{s=1}^{n-1} \VERT \mathsf{F}_n\dots\mathsf{F}_{s+1}\mathsf{F}_s\tilde{\pi}^{x}_{s-1} - \mathsf{F}_n\dots\mathsf{F}_{s+1}\tilde{\mathsf{F}}_s\widetilde{\pi}^{x}_{s-1} \VERT_J
\end{align*}

Following \cite{rebeschini2013can} by application of Lemma \ref{lem:localfilterstability} and Lemma \ref{lemma:corrbounduniformtime} we easily find
$$
	\VERT \pi^{x}_n - \widetilde{\pi}^{x}_n \VERT_J ~\leq~ 8e^{-\beta}(1-\varepsilon^{2\Delta}) \left(|J|e^{-\beta d(J,\partial K_{\sigma(n)}})+\sum_{s=1}^{n-1} e^{-\beta(n-s)}\sum_{v\in J}e^{-\beta d(v, \partial K_{\sigma(n)})} \right)
$$
for every $n\geq 0$, $x\in \mathbb{X}$ and every $K_{\sigma(s)}\in \mathcal{K}_{\sigma(s)}$ such that $J\subseteq K_{\sigma(s)}$ for all $s\in\mathbb{N}$.

For any site $v$ consider the sequence $K_{\sigma(s)}(v) \in \mathcal{K}_{\sigma(s)}$ for all $s\in\mathbb{N}$. The bound just given becomes
$$
	\VERT \pi^{x}_n - \tilde{\pi}^{x}_n \VERT_v ~\leq~ 8e^{-\beta}(1-\varepsilon^{2\Delta})\sum_{s=1}^n e^{-\beta(n-s)}e^{-\beta d(v, \partial K_{\sigma(s)}(v))}
$$
Recall the hypothesis $~\sigma(s)=s~(\mathrm{mod}\,m)$, $s\in\mathbb{N}$. Now averaging the bound gives
$$
	\frac{1}{m}\sum_{k=0}^{m-1} \VERT \pi^{x}_{n-k} - \tilde{\pi}^{x}_{n-k} \VERT_v ~\leq~ 8e^{-\beta}(1-\varepsilon^{2\Delta}) \frac{1}{m}\sum_{k=0}^{m-1} \sum_{s=1}^{n-k} e^{-\beta(n-k-s)}e^{-\beta d(v, \partial K_{\sigma(s)}(v))}
$$
and we define $\VERT \pi^{x}_{-s} - \tilde{\pi}^{x}_{-s} \VERT_v$ for $s\in\mathbb{N}$ to be zero. Now 
\begin{align*}
	\frac{1}{m}\sum_{k=0}^{m-1} \sum_{s=1}^{n-k} e^{-\beta(n-k-s)}e^{-\beta d(v, \partial K_{\sigma(s)}(v))} & ~=~ \frac{1}{m}\sum_{k=0}^{m-1} \sum_{s=0}^{n-k-1} e^{-\beta s}e^{-\beta d(v, \partial K_{\sigma(n-k-s)}(v))} \\
	& ~\leq~ \frac{1}{m}\sum_{k=0}^{m-1} \sum_{s=0}^{n-1} e^{-\beta s}e^{-\beta d(v, \partial K_{\sigma(n-k-s)}(v))}
\end{align*}
where now we extend $~\sigma(s)$ to all $s\in\mathbb{Z}$ so $\sigma(s)=s~(\mathrm{mod}\,m)$. Then
$$
	\frac{1}{m}\sum_{k=0}^{m-1} \sum_{s=0}^{n-1} e^{-\beta s}e^{-\beta d(v, \partial K_{\sigma(n-k-s)}(v))}  ~=~  \frac{1}{m} \sum_{s=0}^{n-1} e^{-\beta s} \sum_{k=0}^{m-1} e^{-\beta d(v, \partial K_{\sigma(n-k-s)}(v))}  ~=~ \frac{1}{m} \sum_{s=0}^{n-1} e^{-\beta s} \sum_{k=0}^{m-1} e^{-\beta d(v, \partial K_{\sigma(k)}(v))}
$$
where we swapped the summation order and used the fact that
$$
	 \sum_{k=0}^{m-1} e^{-\beta d(v, \partial K_{\sigma(n-k-s)}(v))} = \sum_{k=0}^{m-1} e^{-\beta d(v, \partial K_{\sigma(k)}(v))}
$$
for $s=0$ and all $s\in\mathbb{N}$ because of the cyclical partitioning sequence. Now it follows that 
$$
	\frac{1}{m}\sum_{k=0}^{m-1} \VERT \pi^{x}_{n-k} - \tilde{\pi}^{x}_{n-k} \VERT_v ~\leq~ \frac{8e^{-\beta}}{1-e^{-\beta}}(1-\varepsilon^{2\Delta}) \frac{1}{m} \sum_{j=0}^{m-1} e^{-\beta d(v, \partial K_{j}(v))} ~=~ \frac{8e^{-\beta}}{1-e^{-\beta}}(1-\varepsilon^{2\Delta}) \vartheta_m(v)
$$
noting $\sum e^{-\beta s} \leq 1/(1-e^{-\beta})$. Let $\nabla_{d}(v) \triangleq \min_s d(v,\partial K_s)$ then the over bounding
$$
	\tfrac{8e^{-\beta}}{(1-e^{-\beta})}(1-\varepsilon^{2\Delta})\vartheta_m(v) ~\leq~ \tfrac{8e^{-\beta}}{(1-e^{-\beta})}(1-\varepsilon^{2\Delta})\exp\left[-\beta e^{-\beta(\Delta_d(v)-\nabla_d(v))} \frac{1}{m} \theta_m(v) \right]
$$
follows from Slater's inequality and the proof is complete. \qed

Again, both inequalities in Theorem \ref{thm:biasbound} imply that the bias introduced due to blocking can be spatially averaged (or smoothed) across a cyclical application of a sequence of partitions and both inequalities collapse to the result of \cite{rebeschini2013comparison} in the case $m=1$.

\bibliographystyle{plain}
\bibliography{}

\end{document}